\documentclass[11pt]{article}
\usepackage{amsmath}
\usepackage{amssymb}
\usepackage{amsxtra}
\usepackage{amsthm}
\usepackage{mathrsfs}
\usepackage{mathtools}
\usepackage{graphicx}
\usepackage{epstopdf}
\usepackage{subfigure}
\usepackage{tabularx}
\usepackage{array}
\usepackage{float}
\usepackage{xcolor,graphicx,float}
\usepackage{caption}
\usepackage{authblk}
\usepackage[colorlinks,linkcolor=red,anchorcolor=blue,citecolor=blue]{hyperref}
\usepackage{appendix}
\def\EquationsBySection{\def\theequation
	{\thesection.\arabic{equation}}
	\@addtoreset{equation}{section}}

\def \a{{\alpha}}

\def \G{{\Gamma}}
\def \d{{\mathrm{d}}}
\usepackage{indentfirst}
\newtheorem{thm}{Theorem}[section]
\newtheorem{lem}{Lemma}[section]

\newtheorem{rem}{Remark}[section]
\theoremstyle{definition}

\numberwithin{equation}{section}

\newcommand{\be}{\begin{equation}}
	\newcommand{\ee}{\end{equation}}

\newcommand{\bes}{\begin{equation*}}
	\newcommand{\ees}{\end{equation*}}

\renewcommand{\theequation}{\thesection.\arabic{equation}}

\newcommand{\bean}{\begin{eqnarray*}}
	\newcommand{\eean}{\end{eqnarray*}}

\newcommand\old[1]{}
\makeatother
\usepackage{appendix}
\usepackage{geometry}
\allowdisplaybreaks[3]
\EquationsBySection
\begin{document}
	\date{}
	\title{Volterra type McKean-Vlasov SDEs with singular kernels: Well-posedness, Propagation of Chaos and Euler schemes\thanks{S. Liu is supported by Postgraduate Research  Practice
			Innovation Program of Jiangsu Province (No. KYCX23 1668). H. Gao is supported in part by the NSFC Grant Nos. 12171084 and  the fundamental Research
			Funds for the Central Universities No. RF1028623037.}}
	\author{Shanqi Liu\footnote{School of Mathematical Science, Nanjing Normal University, Nanjing 210023, China} \thanks{E-mail: shanqiliumath@126.com} and Hongjun Gao\footnote{School of Mathematics, Southeast University, Nanjing 211189, China} \thanks{Corresponding author. E-mail: hjgao@seu.edu.cn}}
	\maketitle
	\noindent {\bf \small Abstract}{\small
		\quad In this paper, our work is devoted to studying Volterra type McKean-Vlasov  stochastic differential equations with singular kernels. Firstly, the well-posedness of Volterra type McKean-Vlasov stochastic differential equations are established. And then propagation of chaos is proved with explicit
		estimate of the convergence rate.  Finally, We also propose an explicit Euler scheme for an interacting particle system associated with the Volterra type
		McKean-Vlasov equation.}


	\newcommand{\etalchar}[1]{$^{#1}$}
\providecommand{\bysame}{\leavevmode\hbox to3em{\hrulefill}\thinspace}
\providecommand{\MR}{\relax\ifhmode\unskip\space\fi MR }
\providecommand{\MRhref}[2]{%
	\href{http://www.ams.org/mathscinet-getitem?mr=#1}{#2}
}
\providecommand{\href}[2]{#2}
\section{Introduction}
	Stochastic differential equations (SDEs) are mathematical models that characterize complex phenomena including physical, chemical, and biological sciences. Among them, Volterra-type stochastic differential equations are important and interesting class of research objects with a large number of applications in mathematics finance, biology, etc. In particular, when the kernel $K(t,s)=K^H(t,s)$, where $K^H(t,s)$ is defined in Remark \ref{exp}, Volterra-type SDEs is reduced to the classical SDEs driven by fractional Brownian motion. {\bf If the kernel $K(t,s)=(t-s)^{-\alpha}$ with suitable index $\alpha$, Volterra-type SDEs is reduced to Caputo fractional SDEs, see e.g. \cite{AMN}}.

Stochastic Volterra integral equations are more complex models compared with usual SDEs. On the one hand, these equations are in general neither Markovian nor semi-martingales, and include fractional Brownian motion with Hurst index smaller than $1/2$ as a special case. On the other hand, possibly singular kernels need to be controlled carefully, which causes fundamental technical difficulties. In the past decade, stochastic Volterra integral equations have been intensively investigated, e.g. the well-posedness \cite{ALP,AKO,BM,PP,Pro,PS,PS1,SWY,Wan,WH23,You1,You2,ZX2}, the asymptotic properties such as the
large deviation \cite{Lak,JP,NR,ZX1,ZX2}, moderate deviation \cite{LWYZ,JP} and central limit theorem \cite{Qiao}, and numerical methods \cite{LHH2,RTY,ZX1}.

McKean-Vlasov SDEs (also called distribution dependent SDEs \cite{RW} or mean-field SDEs), it was first discussed by Kac in literatures \cite{Kac56, Kac59}, which describe limiting behaviors of individual particles in an interacting particle system of mean-field type when the number of particles goes to infinity (so called propagation of chaos) and widely used in various fields. One can refer to \cite{LSZZ} and the references therein for more details. Volterra-type McKean-Vlasov SDEs first appeared in the paper \cite{SWY}, which
studied the well-posedness and Pontryagin's type maximum principle of a class
mean-field backward stochastic {V}olterra integral equations. Recently, there has been an growing interest in the
study of mean-field forward (backward) SDEs. In \cite{WH23}, the authors extended
the results in \cite{SWY} by considering the well-posedness and Pontryagin's type
maximum principle of the mean-field backward doubly stochastic Volterra integral
equation. Fan and Jiang \cite{FJ23} studied a class of distribution-dependent SDEs
driven by Volterra processes and the associated optimal control problem.

To the best of our knowledge, it seems that there are few works on forward Volterra type McKean-Valsov SDEs. For the regular kernel, Jie et al. \cite{JLZ} studied one-dimensional {M}c{K}ean--{V}lasov stochastic {V}olterra
equations with {H}\"{o}lder diffusion coefficients. One can also refer Theorem $2.10$ in \cite{PS1}. The aim of the paper is to study basic properties of Volterra type McKean-Vlasov SDEs with singular kernels. More precisely, we first consider the well-posedness and propagation of chaos for Volterra type McKean-Vlasov SDEs under Lipschitz condition in the sense of distribution measure. Very recently, we remark that Pr\"{o}mel and Scheffels \cite{PS1} also consider the same problems for Volterra type McKean-Vlasov SDEs with singular kernels. Compared with Theorem $3.2$ in \cite{PS1}, our assumption $\mathbf{(H1')}$ is weaker than assumption $2.1$ in \cite{PS1}. Volterra type Gronwall's inequality play the key role, which provides a clear estimation framework. As an application, we propose an explicit Euler scheme for an interacting particle system associated with the Volterra type
McKean-Vlasov equation with singular kernels inspired by \cite{ZX1}.

The rest of the paper is organized as follows. In Section
\ref{Preliminaries}, we review some basic notations, and a Gronwall's inequality of Volterra type is introduced. In Section \ref{Well-posedness for Volterra type McKean-Vlasov SDEs}, under some suitable conditions, the existence, uniqueness and path continuity of solution to Volterra type McKean-Vlasov stochastic differential equations are established. Later in Section \ref{Propagation of chaos}, we discuss Propagation of chaos for Volterra type McKean-Vlasov SDEs. Finally in Section \ref{Euler schemes for Volterra type McKean-Vlasov SDEs}, an explicit Euler scheme for an interacting particle system associated with the Volterra type
McKean-Vlasov equation with singular kernels is constructed.
\section{Preliminary}\label{Preliminaries}
\subsection{Notations}
Throughout this paper, we denote $\mathbb{R}^{d}$ the $n$-dimensional Euclidean space. Let $(\Omega,\mathcal{F},(\mathcal{F}_t)_{t\ge0},\mathbb{P})$ to be a given complete filtered probability space $(\Omega,\mathcal{F},(\mathcal{F}_t)_{t\ge0},\mathbb{P})$, where $\mathcal{F}_t$ is a nondecreasing family of sub-$\sigma$ fields of $\mathcal{F}$ satisfying the usual conditions. Let $\mathcal{P}$ be the space of all probability measures $\mu$ on $\mathbb{R}^{d}$, and let
$$\mathcal{P}_{2}(\mathbb{R}^{d})=\left\{\mu \in \mathcal{P}\left(\mathbb{R}^{d}\right): \mu\left(|\cdot|^{2}\right):=\int_{\mathbb{R}^{d}}|x|^{2} \mu(\mathrm{d} x)<\infty\right\}.$$
It is well known that $\mathcal{P}_{2}$ is a Polish space under the Wasserstein distance
$$
\mathbb{W}_{2}(\mu, \nu):=\inf _{\pi \in \mathcal{C}(\mu, \nu)}\left(\int_{\mathbb{R}^{m} \times \mathbb{R}^{m}}|x-y|^{2} \pi(\mathrm{d} x, \mathrm{d} y)\right)^{\frac{1}{2}},\ \mu,\nu \in \mathcal{P}_{2}\left(\mathbb{R}^{d}\right),
$$
where $\mathcal{C}(\mu, \nu)$ is the set of couplings for $\mu$ and $\nu$; that is, $\pi \in \mathcal{C}(\mu, \nu)$ is a probability measure on $\mathbb{R}^{d} \times \mathbb{R}^{d}$ such that $\pi\left(\cdot \times \mathbb{R}^{d}\right)=\mu$ and $\pi\left(\mathbb{R}^{d} \times \cdot\right)=\nu$.

For a measurable function $K$ on $\mathbb{R}_{+}\times \mathbb{R}_{+}$ and measurable function $L$ on $\mathbb{R}_{+}\times \mathbb{R}_{+}$, the convolutions $K*L$ and $L*K$ are defined by
$$(K*L)(t,s)=\int_{s}^{t}K(t,u)L(u,s)\d u,$$
$$(L*K)(t,s)=\int_{s}^{t}L(t,u)K(u,s)\d u,$$
for almost all $(t,s)\in \Delta$, where $\Delta:=\{t,s\in\mathbb{R}_{+}\times \mathbb{R}_{+}: s\leq t\}.$
\subsection{Growall's inequality of Volterra type}
We first recall the following lemma due to Girpenberg \cite{GG} (Theorem $1$).
\begin{lem}\label{lemma 2.1}
	Let $K:\Delta\to\mathbb{R}_{+}$ be a measurable function. Assume that for any $T>0$,
	$$t\mapsto \int_{0}^{t}K(t,s)\d s\in L^{\infty}(0,T)$$
	and
	$$\limsup_{\varepsilon\downarrow 0}\Big\|\int_{\cdot+\varepsilon}^{\cdot}K(\cdot+\varepsilon,s)\d s\Big\|_{L^{\infty}(0,T)}<1.$$
	Define
	\begin{align}\label{definition of R_n}
		R_1(t,s):=K(t,s),\quad R_{n+1}(t,s):=\int_{s}^{t}K(t,u)R_{n}(u,s)\d u, n\in\mathbb{N}.	
	\end{align}
	Then for any $T>0$, there exist constants $C_{T}>0$ and $\gamma\in (0,1)$ such that
	$$\Big\|\int_{0}^{\cdot}R_{n}(\cdot,s)ds\Big\|_{L^{\infty}(0,T)}\leq C_{T}n\gamma^{n}, \forall n\in \mathbb{N}.$$
	In particular, the series
	\begin{align}\label{Resolvent}
		R(t,s):=\sum_{n=1}^{\infty}R_{n}(t,s)
	\end{align}
	converges for almost all $(t,s)\in \Delta$, and
	\begin{align}\label{definition of resolvent}
		R(t,s)-K(t,s)=(K*R)(t,s)-(R*K)(t,s)
	\end{align}
	and for any $T>0$,
	\begin{align}\label{finite of resolvent}
		t\mapsto \int_{0}^{t}R(t,s)\d s\in L^{\infty}(0,T).
	\end{align}
\end{lem}
The function $R$ defined by \eqref{Resolvent} is called the resolvent of the kernel $K$. All the functions $K$ in Lemma \ref{lemma 2.1} will be denoted by $\mathcal{K}$. We also denote by $\mathcal{K}_{>1}$ the set of all nonnegative measurable functions $K$ on $\Delta$ with the property that for any $T>0$ and some $\beta>1$,
$$t\mapsto \int_{0}^{t}K^{\beta}(t,s)\d s\in L^{\infty}(0,T).$$
It is clear that $\mathcal{K}_{>1}\subset \mathcal{K}.$
We now present Gronwall's lemma of Volterra type due to Zhang \cite{ZX2} (Lemma $2.2$).
\begin{lem}\label{Gronwall}
	Let $K\in\mathcal{K}$ and $R_n$ and $R$ be defined by \eqref{definition of R_n} and \eqref{Resolvent}, respectively. Let $f,g:\mathbb{R}_{+}\to\mathbb{R}_{+}$ be two measurable functions satisfying that for any $T>0$ and some $n\in\mathbb{N}$,
	$$t\mapsto \int_{0}^{t}R_n(t,s)f(s)\d s\in L^{\infty}(0,T)$$
	and for almost all $t\in(0,\infty)$,
	$$\int_{0}^{t}R(t,s)g(s)\d s<\infty.$$
	If for almost all $t\in (0,\infty)$,
	$$f(t)\leq g(t)+\int_{0}^{t}K(t,s)f(s)\d s,$$
	then for almost all $t\in (0,\infty)$,
	$$f(t)\leq g(t)+\int_{0}^{t}R(t,s)g(s)\d s.$$
\end{lem}
\section{Well-posedness for Volterra type McKean-Vlasov SDEs}\label{Well-posedness for Volterra type McKean-Vlasov SDEs}
\subsection{Global existence and uniqueness}
Consider the following Volterra type McKean-Vlasov SDEs:
\begin{align}\label{Volterra eq}
	X_t=X_0+\int_{0}^{t}b(t,s,X_s,\mathcal{L}_{X_s})\d s+\int_{0}^{t}\sigma(t,s,X_s,\mathcal{L}_{X_s})\d W_s,
\end{align}
where $X_0\in\mathbb{R}^d$ and $\sigma:\mathbb{R}_{+}\times \mathbb{R}_{+}\times\mathbb{R}^d\times \mathcal{P}_2(\mathbb{R}^d)\to\mathbb{R}^d\times \mathbb{R}^m$, $b:\mathbb{R}_{+}\times \mathbb{R}_{+}\times\mathbb{R}^d\times \mathcal{P}_2(\mathbb{R}^d)\to\mathbb{R}^d$ are Borel measurable functions, and $\{W_t\}_{t\geq0}$ is an m-dimensional standard Brownian motion definied on the classical Wiener space.

Now, we introduce the hypothesis under which
we will be able to prove the global existence and uniqueness of solutions to \eqref{Volterra eq}.

$\mathbf{(H1)}$ Integrability condition:
$$ \sup_{t\in [0,T]}\int_{0}^{t}\Big[K_1(t,s)+K_2(t,s)+K_3(t,s)+K_4(t,s)\Big]\d s\leq C_T, \quad t\in [0,T].$$
where $K_i, i=1,2,3,4$ are from $\mathbf{(H2)}$ and $\mathbf{(H3)}$ below.

$\mathbf{(H2)}$ Lipschitz condition:
There exist positive functions $K_1(t,s)\in\mathcal{K}$ and $K_2(t,s)\in\mathcal{K}$ such that for all $x,y\in\mathbb{R}^d$, $\mu,\nu\in \mathcal{P}_{2}(\mathbb{R}^d)$ and $t,s\in [0,T]$
\begin{align}
	|b(t,s,x,\mu)-b(t,s,y,\nu)|&\leq K_1(t,s)\Big(|x-y|+\mathbb{W}_{2}(\mu,\nu)\Big),\nonumber\\||\sigma(t,s,x,\mu)-\sigma(t,s,y,\nu)||^2&\leq K_2(t,s)\Big(|x-y|^2+\mathbb{W}_{2}^2(\mu,\nu)\Big).\nonumber
\end{align}

$\mathbf{(H3)}$ Linear growth condition:
There exist positive functions $K_3(t,s)\in\mathcal{K}$ and $K_4(t,s)\in\mathcal{K}$ such that for all $x\in\mathbb{R}^d$, $\mu\in \mathcal{P}_{2}(\mathbb{R}^d)$ and $t,s\in [0,T]$
\begin{align}
	|b(t,s,x,\mu)|&\leq K_3(t,s)\Big(1+|x|+\mathbb{W}_{2}(\mu,\delta_0)\Big),\nonumber\\||\sigma(t,s,x,\mu)||^2&\leq K_4(t,s)\Big(1+|x|^2+\mathbb{W}_{2}^2(\mu,\delta_0)\Big).\nonumber
\end{align}
\begin{thm}\label{existence}
	Assume that $X_0\in L^p(\Omega,\mathcal{F}_0,P)$ and assumptions $\mathbf{(H1)}$-$\mathbf{(H3)}$ hold, Then there exists a unique measurable $(\mathcal{F}_t)$-adapted process $X_t$ such that for almost all $t\geq0$,
	$$X_t=X_0+\int_{0}^{t}b(t,s,X_s,\mathcal{L}_{X_s})\d s+\int_{0}^{t}\sigma(t,s,X_s,\mathcal{L}_{X_s})\d W_s, t\in [0,T], P\text{-}a.s.,$$
	and for any $T>0$ some constant $C_{p,T}$
	\begin{align}\label{moment estimate of X}
		E|X_t|^p\leq C_{p,T}(E|X_0|^p+1).
	\end{align}
	for almost all $t\in [0,T]$.
\end{thm}
\begin{proof}
	We utilize the successive approximation scheme to prove Theorem \ref{existence}. Define recursively $(X^n)_{n\geq 1}$ as follows: $X^1_t:=X_0$, $t\in [0,T]$ and for each $n\in\mathbb{N}$,
	\begin{align}\label{Picard}
		X^{n+1}_t:=X_0+\int_{0}^{t}b(t,s,X^{n}_s,\mathcal{L}_{X^{n}_s})\d s+\int_{0}^{t}\sigma(t,s,X^{n}_s,\mathcal{L}_{X^{n}_s})\d W_s, t\in [0,T].
	\end{align}
	By Burkholder-Davis-Gundy (BDG) inequality, H$\mathrm{\ddot{o}}$lder's inequality, $\mathbb{W}_{2}(\mathscr{L}_{X}, \mathscr{L}_{Y})\leq [E|X-Y|^2]^{\frac{1}{2}}$ and $\mathbf{(H1)}$, $\mathbf{(H3)}$ we have
	\begin{align}\label{estimate of X_n}
		E\Big(|X_t^{n+1}|^p\Big)&\leq 3^{p-1}E|X_0|^p	+3^{p-1}E\Big(\Big|\int_{0}^{t}b(t,s,X^{n}_s,\mathcal{L}_{X^{n}_s})\d s\Big|^p\Big)\nonumber\\&+3^{p-1}E\Big(\Big|\int_{0}^{t}\sigma(t,s,X^{n}_s,\mathcal{L}_{X^{n}_s})\d W_s\Big|^p\Big)\nonumber\\&\leq 3^{p-1}E|X_0|^p +3^{p-1}E\Big(\int_{0}^{t}K_3(t,s)\big|1+|X^{n}_s|+\mathbb{W}_{2}(\mathcal{L}_{X^{n}_s},\delta_0)\big|\d s\Big)^p\nonumber\\&+C_p E\Big(\int_{0}^{t}||\sigma(t,s,X^{n}_s,\mathcal{L}_{X^{n}_s})||^2\d s\Big)^{\frac{p}{2}}\nonumber\\&\leq 3^{p-1}E|X_0|^p +C_p\int_{0}^{t}K_3(t,s)\cdot E\Big(1+|X^{n}_s|^p+\mathbb{W}^p_{2}(\mathcal{L}_{X^{n}_s},\delta_0)\Big)\d s\cdot\Big(\int_{0}^{t}K_3(t,s)\d s\Big)^{p-1}\nonumber\\&+ C_p\int_{0}^{t}K_4(t,s)\cdot E\Big(1+|X^{n}_s|^p+\mathbb{W}^p_{2}(\mathcal{L}_{X^{n}_s},\delta_0)\Big)\d s\cdot\Big(\int_{0}^{t}K_3(t,s)\d s\Big)^{\frac{p-2}{2}}\nonumber\\&\leq 3^{p-1}E|X_0|^p+C_{p,T}\int_{0}^{t}K_3(t,s)\cdot E(1+|X_s^n|^p)\d s+C_{p,T}\int_{0}^{t}K_4(t,s)\cdot E(1+|X_s^n|^p)\d s\nonumber\\&\leq 3^{p-1}E|X_0|^p+C_{p,T}+C_{p,T}\int_{0}^{t}\Big(K_3(t,s)+K_4(t,s)\Big)\cdot E|X_s^n|^p \d s\nonumber\\&\leq C_{p,T}E|X_0|^p+C_{p,T}\int_{0}^{t}\Big(K_3(t,s)+K_4(t,s)\Big)\cdot E|X_s^n|^p \d s.
	\end{align}
	Let
	$$f_m(t):=\sup_{n=1,\cdots,m}E|X_s^n|^p,$$
	then we get
	$$f_m(t)\leq C_{p,T}(E|X_0|^p+1)+\int_{0}^{t}K_{3,4}(t,s)\cdot f_m(s)\d s,\quad \forall m\geq1,$$
	where $K_{3,4}:=C_{p,T}\cdot (K_3(t,s)+K_4(t,s))$ and the constant $C_{p,T}$ is independent of $m$.
	
	Let $R^{K_{3,4}}$ be defined by \eqref{Resolvent} in terms of $K_{3,4}$. Note that by \eqref{definition of R_n}, \eqref{definition of resolvent}, \eqref{finite of resolvent} and $\mathbf{(H1)}$, when $n=1$ we have
	\begin{align}
		\int_{0}^{T}R_1^{K_{3,4}}(t,s)f_m(s)\d s=\int_{0}^{T}K_{3,4}(t,s)f_m(s)\d s<\infty,\nonumber
	\end{align}
	and
	\begin{align}
		&\int_{0}^{t}R^{K_{3,4}}(t,s)\cdot E|X_0|^p\d s\nonumber\\&=\int_{0}^{t}K_{3,4}(t,s)\cdot E|X_0|^p\d s+\int_{0}^{t}\Big(\int_{s}^{t}R^{K_{3,4}}(t,u)K_{3,4}(u,s)\d u\Big)\cdot E|X_0|^p\d s\nonumber\\&=\int_{0}^{t}K_{3,4}(t,s)\d s\cdot E|X_0|^p+\int_{0}^{t}R^{K_{3,4}}(t,u)\Big(\int_{0}^{u}K_{3,4}(u,s)\d s\Big)\d u\cdot E|X_0|^p<\infty.\nonumber
	\end{align}
	Hence, Combining \eqref{definition of resolvent}, \eqref{finite of resolvent}, $\mathbf{(H1)}$ and Lemma \ref{Gronwall}, we obtain that for almost all $t\in [0,T]$,
	\begin{align}\label{bounded}
		\sup_{n\in\mathbb{N}}E|X^n_t|^p&=\lim_{m\to\infty}f_m(t)\leq C_{p,T}(E|X_0|^p+1)+\int_{0}^{t}R^{K_{3,4}}(t,s)\cdot (E|X_0|^p+1)\d s\Big)\nonumber\\&\leq C_{p,T}(E|X_0|^p+1).
	\end{align}
	On the other hand, we set
	$$Z_{n,m}(t):=X^n_t-X^m_t$$
	and
	$$f(t):=\limsup_{n,m\to\infty}E|Z_{n,m}(t)|^p.$$
	Applying BDG's inequality, H$\mathrm{\ddot{o}}$lder's inequality, $\mathbf{(H1)}$ and $\mathbf{(H2)}$ we obtain
	\begin{align}
		E|Z_{n+1,m+1}(t)|^p&=E|X^{n+1}_t-X^{m+1}_t|^p\nonumber\\&\leq 2^{p-1}E\Big(\Big|\int_{0}^{t}|b(t,s,X^{n}_s,\mathcal{L}_{X^{n}_s})-b(t,s,X^{m}_s,\mathcal{L}_{X^{m}_s})|\d s\Big|^p\Big)\nonumber\\&+2^{p-1}E\Big(\Big|\int_{0}^{t}\big(\sigma(t,s,X^{n}_s,\mathcal{L}_{X^{n}_s})-\sigma(t,s,X^{m}_s,\mathcal{L}_{X^{m}_s})\big)\d W_s\Big|^p\Big)\nonumber\\&\leq 2^{p-1}E\Big(\int_{0}^{t}K_1(t,s)\cdot\big(\big|X^{n}_s-X^{m}_s|+\mathbb{W}_{2}(\mathcal{L}_{X^{n}_s},\mathcal{L}_{X^{m}_s})\big)\d s\Big)^p\nonumber\\&+C_p E\Big(\int_{0}^{t}||\sigma(t,s,X^{n}_s,\mathcal{L}_{X^{n}_s})-\sigma(t,s,X^{m}_s,\mathcal{L}_{X^{m}_s})||^2\d s\Big)^{\frac{p}{2}}\nonumber\\&\leq  C_p\int_{0}^{t}K_1(t,s)\cdot E\Big(\big|X^{n}_s-X^{m}_s\big|^p+\mathbb{W}^p_{2}(\mathcal{L}_{X^{n}_s},\mathcal{L}_{X^{m}_s})\Big)\d s\cdot\Big(\int_{0}^{t}K_1(t,s)\d s\Big)^{p-1}\nonumber\\&+ C_p\int_{0}^{t}K_2(t,s)\cdot E\Big(\big|X^{n}_s-X^{m}_s\big|^p+\mathbb{W}^p_{2}(\mathcal{L}_{X^{n}_s},\mathcal{L}_{X^{m}_s})\Big)\d s\cdot\Big(\int_{0}^{t}K_2(t,s)\d s\Big)^{\frac{p-1}{2}}\nonumber\\&\leq C_{p,T}\int_{0}^{t}K_1(t,s)\cdot E|X_s^n-X_s^m|^p\d s+C_{p,T}\int_{0}^{t}K_2(t,s)\cdot E|X_s^n-X_s^m|^p\d s\nonumber\\&\leq C_{p,T}\int_{0}^{t}\Big(K_1(t,s)+K_2(t,s)\Big)\cdot E|Z_{n,m}(t)|^p \d s.\nonumber
	\end{align}
	By \eqref{bounded}, $\mathbf{(H1)}$ and applying Fatou's lemma, we deduce that
	$$f(t)\leq C_{p,T}\int_{0}^{t}\Big(K_1(t,s)+K_2(t,s)\Big)\cdot f(s) \d s.$$
	Using Gronwall's inequality of Volterra type Lemma \ref{Gronwall}, we get for almost all $t\in [0,T]$,
	$$f(t)=\limsup_{n,m\to\infty}E|Z_{n,m}(t)|^p=0.$$
	Hence, there exists a $\mathbb{R}^d$-valued $(\mathcal{F}_t)$-adapted process $X_t$ such that for almost all $t\in [0,T]$,
	$$\lim_{m\to\infty}E|X^n_t-X_t|^p=0.$$
	Now taking limits for \eqref{Picard} gives the existence. The uniqueness follows from a similar discussion.
\end{proof}
\subsection{Path continuity of solutions}
In this subsection, in order to study path continuity of solution to \eqref{Volterra eq}, in addition to $(\mathbf{H2})$-$(\mathbf{H3})$, we also assume that:

$\mathbf{(H1')}$
There exists a positive constant $\gamma>0$ such that for any $t,t'\in [0,T]$
$$\int_{0}^{t}|K_i(t',s)-K_i(t,s)|\d s\leq C_T|t-t'|^{\gamma}, \quad t\in [0,T],$$
$$\int_{0}^{t}|K_j(t',s)-K_j(t,s)|^2\d s\leq C_T|t-t'|^{2\gamma}, \quad t\in [0,T],$$
$$\int_{t}^{t'}|K_i(t,s)|^2\d s\leq C_T|t-t'|^{2\gamma}, \quad t\in [0,T],$$
$$\int_{t}^{t'}|K_j(t,s)|^2\d s\leq C_T|t-t'|^{2\gamma}, \quad t\in [0,T].$$
where $i=1,3, j=2,4$.

$\mathbf{(H4)}$ For all $t,t',s\in [0,T]$, $\mu\in \mathcal{P}_2(\mathbb{R}^d)$ and $x\in \mathbb{R}^{d}$,
\begin{align}
	|b(t',s,x,\mu)-b(t,s,x,\mu)|&\leq F_1(t',t,s)\Big(1+|x|+\mathbb{W}_2(\mu,\delta_0)\Big),\nonumber\\||\sigma(t,s,x,\mu)-\sigma(t,s,y,\mu)||^2&\leq F_2(t',t,s)\Big(1+|x|^2+\mathbb{W}_2^2(\mu,\delta_0)\Big),\nonumber
\end{align}
where $F_i(t',t,s), i=1,2,$ are positive function on $[0,T]\times[0,T]\times[0,T]$, and satisfy for some $\gamma>0$
$$\int_{0}^{t}F_1(t',t,s)\d s\leq C_T|t-t'|^{\gamma},\quad \int_{0}^{t}F_2(t',t,s)\d s\leq C_T|t-t'|^{2\gamma}.$$
\begin{rem}
	It is clear that $\mathbf{(H1')}$ implies $K_k\in \mathcal{K}, k=1,2,3,4$. Indeed, $$t\mapsto \int_{0}^{t}K_k(t,s)\d s\in L^{\infty}(0,T)$$
	and
	$$\limsup_{\varepsilon\downarrow 0}\Big\|\int_{\cdot+\varepsilon}^{\cdot}K_k(\cdot+\varepsilon,s)\d s\Big\|_{L^{\infty}(0,T)}<1.$$
	So under stronger assumptions $\mathbf{(H1')}$, $\mathbf{(H2)}$ and $\mathbf{(H3)}$, we still can prove the results as Theorem \ref{existence}. Moreover, compared with Theorem $3.2$ in \cite{PS1}, our assumption $\mathbf{(H1')}$ is weaker than assumption $2.1$ in \cite{PS1}.
\end{rem}
\begin{rem}\label{exp}
	Assume $K_1=K_2=K_3=K_4=K$, let us list some examples of kernels that satisfy $\mathbf{(H1')}$:
	
	$\bullet$ fractional Brownian motion kernel. The fractional Brownian motion has the integral representation in law:
	\begin{align}
		B_t^H=\int_{0}^{1}K^H(t,s)dW_s,\nonumber
	\end{align}
	where $W$ is a standard Brownian motion and $K^H$ is the square integral kernel:
	\begin{align}
		K^H(t,s)=K(t,s)=c_H(t-s)^{H-\frac{1}{2}}+c_H(\frac{1}{2}-H)\int_{s}^{t}(\theta-s)^{H-\frac{3}{2}}\Big(1-(\frac{s}{\theta})^{\frac{1}{2}-H}\Big)\mathrm{d}\theta,\nonumber
	\end{align}
	with
	$$c_H=\Big(\frac{2H\G(\frac{3}{2}-H)}{\G(H+\frac{1}{2})\G(2-2H)}\Big)^{\frac{1}{2}}.$$
	
	It was proved in \cite{AMN} that for singular case $(H<\frac{1}{2})$ and regular case $(H>\frac{1}{2})$ we have $(\alpha=|H-\frac{1}{2}|)$
	\begin{align}
		\text{Singular case}\quad\big|\frac{\partial K^H}{\partial t}(t,s)\big|\leq C_{\alpha}(t-s)^{-\alpha-1}\ \text{and} \int_{s}^{t}K^H(t,u)^2\mathrm{d}u\leq C_{\alpha}(t-s)^{1-2\alpha},\ s<t.\nonumber
	\end{align}
	and
	\begin{align}
		\text{Regular case}\quad\big|\frac{\partial K^H}{\partial t}(t,s)\big|\leq C_{\alpha}(t-s)^{\alpha-1}s^{-\alpha}\ \text{and} \int_{s}^{t}K^H(t,u)^2\mathrm{d}u\leq C_{\alpha}(t-s)^{1+2\alpha},\ s<t.\nonumber
	\end{align}
	So for the singular case and regular case we have that $(t<t')$
	$$\int_{t}^{t'}(K^H(t,s))^{2}\mathrm{d}s\leq C_{\alpha}(t'-t)^{1-2\alpha}=C_{\alpha}(t'-t)^{2H},$$
	and
	$$\int_{t}^{t'}(K^H(t,s))^{2}\mathrm{d}s\leq C_{\alpha}(t'-t)^{1+2\alpha}=C_{\alpha}(t'-t)^{2H}.$$
	Moreover, it was proved in \cite{ZX1} that for $H\in (0,1)$
	$$\int_{0}^{t}\big|K^H(t',s)-K^H(t,s)\big|^2\mathrm{d}s\leq |t'-t|^{2H}.$$
	Therefore, it is clear to verify $\mathbf{(H1')}$ hold with $\gamma=2H$.
	
	$\bullet$ Convolutional kernel, for the details see Example $2.3$ in \cite{ALP}. Note that our assumption $\mathbf{(H1')}$ covers the conditions $(2.5)$ in \cite{ALP}.
\end{rem}
\begin{thm}
	Assume that $X_0\in L^p(\Omega,\mathcal{F}_0,P)$, $p>2\vee \frac{2}{\gamma}$ and assumptions $\mathbf{(H1')}$, $\mathbf{(H2)}$-$\mathbf{(H4)}$ hold. Then $t\mapsto X(t)$ admits a $\a$-order H$\mathrm{\ddot{o}}$lder continuous modification for any $\delta\in [0,\frac{\gamma}{2}-\frac{1}{p})$.
\end{thm}
\begin{proof}
	For any $0\leq t<t'\leq T$, we first rewrite $X_{t'}-X_t$ as
	\begin{align}
		&X_{t'}-X_t\nonumber\\&=\int_{t}^{t'}b(t',t,X_s,\mathcal{L}_{X_s})\d s+\int_{0}^{t}\big(b(t',s,X_s,\mathcal{L}_{X_s})-b(t,s,X_s,\mathcal{L}_{X_s})\big)\d s\nonumber\\&+\int_{t}^{t'}\sigma(t',t,X_s,\mathcal{L}_{X_s})\d W_s+\int_{0}^{t}\big(\sigma(t',s,X_s,\mathcal{L}_{X_s})-\sigma(t,s,X_s,\mathcal{L}_{X_s})\big)\d W_s\nonumber\\&:=I_1(t',t)+I_2(t',t)+I_3(t',t)+I_4(t',t).\nonumber
	\end{align}
	As for the term $I_1(t',t)$,  applying H$\mathrm{\ddot{o}}$lder's inequality,  $\mathbf{(H1')}$, $\mathbf{(H3)}$ and \eqref{moment estimate of X}, we get
	\begin{align}
		E|I_1(t',t)|^p&\leq E\Big|\int_{t}^{t'}K_3(t,s)\cdot\Big(1+|X_s|+\mathbb{W}_{2}(\mathcal{L}_{X_s},\delta_0)\Big)\d s\Big|^p\nonumber\\&\leq E\int_{t}^{t'}K_3(t,s)\cdot\Big|1+|X_s|+\mathbb{W}_{2}(\mathcal{L}_{X_s},\delta_0)\Big|^p\d s\cdot\Big(\int_{t}^{t'}K_3(t,s)\d s\Big)^{p-1}\nonumber\\&\leq C_{p,T}|t-t'|^{\gamma(p-1)}\int_{t}^{t'}K_3(t,s)\cdot\big(1+E|X_s|^p\big)\d s\nonumber\\&\leq C_{p,T}\big(1+E|X_0|^p\big)|t-t'|^{\gamma p}.\nonumber
	\end{align}
	As for the term $I_2(t',t)$,  by H$\mathrm{\ddot{o}}$lder's inequality,  $\mathbf{(H1')}$, $\mathbf{(H4)}$ and \eqref{moment estimate of X}, we have
	\begin{align}
		E|I_2(t',t)|^p&\leq E\Big|\int_{t}^{t'}F_1(t',t,s)\cdot\Big(1+|X_s|+\mathbb{W}_{2}(\mathcal{L}_{X_s},\delta_0)\Big)\d s\Big|^p\nonumber\\&\leq E\int_{t}^{t'}F_1(t',t,s)\cdot\Big|1+|X_s|+\mathbb{W}_{2}(\mathcal{L}_{X_s},\delta_0)\Big|^p\d s\cdot\Big(\int_{t}^{t'}F_1(t',t,s)\d s\Big)^{p-1}\nonumber\\&\leq C_{p,T}|t-t'|^{\gamma(p-1)}\int_{t}^{t'}F_1(t',t,s)\cdot\big(1+E|X_s|^p\big)\d s\nonumber\\&\leq C_{p,T}\big(1+E|X_0|^p\big)|t-t'|^{\gamma p}.\nonumber
	\end{align}
	As for the term $I_3(t',t)$, by BDG's inequality, H$\mathrm{\ddot{o}}$lder's inequality, $\mathbf{(H1')}$, $\mathbf{(H3)}$ and \eqref{moment estimate of X}, for $p> 2$ we have
	\begin{align}
		E|I_3(t',t)|^p&\leq C_p E\Big(\int_{t}^{t'}\|\sigma(t',s,X_s,\mathcal{L}_{X_s})\|^2\d s\Big)^{\frac{p}{2}}\nonumber\\&\leq C_pE\Big(\int_{t}^{t'}K_4(t',s)\cdot\Big(1+|X_s|^2+\mathbb{W}_{2}^2(\mathcal{L}_{X_s},\delta_0)\Big)\d s\Big)^{\frac{p}{2}}\nonumber\\&\leq C_{p}E\int_{t}^{t'}K_4(t',s)\cdot\Big(1+|X_s|^2+\mathbb{W}_{2}^2(\mathcal{L}_{X_s},\delta_0)\Big)^p\d s\cdot\Big(\int_{t}^{t'}K_4(t',s)\d s\Big)^{\frac{p}{2}-1}\nonumber\\&\leq C_{p,T}|t-t'|^{p-2}\int_{t}^{t'}K_4(t',s)\cdot\Big(1+E|X_s|^p\Big)\d s\nonumber\\&\leq C_{p,T}\big(1+E|X_0|^p\big)|t-t'|^{\gamma p}.\nonumber
	\end{align}
	As for the term $I_4(t',t)$, by BDG's inequality, H$\mathrm{\ddot{o}}$lder's inequality, extended Minkowski's inequality (Corollary $1.30$ in \cite{KO}), $\mathbf{(H4)}$ and \eqref{moment estimate of X}, we obtain
	\begin{align}
		E|I_4(t',t)|^p&\leq C_pE\Big(\int_{0}^{t}\|\sigma(t',s,X_s,\mathcal{L}_{X_s})-\sigma(t,s,X_s,\mathcal{L}_{X_s})\|^2\d s\Big)^{\frac{p}{2}}\nonumber\\&\leq C_pE\Big(\int_{0}^{t}F_2(t',t,s)\cdot\Big(1+|X_s|^2+\mathbb{W}_2^2(\mathcal{L}_{X_s},\delta_0)\Big)\d s\Big)^{\frac{p}{2}}\nonumber\\&\leq C_p\Big(\int_{0}^{t}F_2(t',t,s)\cdot \big(E\big[1+|X_s|^p+\mathbb{W}_2^p(\mathcal{L}_{X_s},\delta_0)\big]\big)^{\frac{2}{p}}\d s\Big)^{\frac{p}{2}}\nonumber\\&\leq C_p\Big(\int_{0}^{t}F_2(t',t,s)\cdot \big(1+(E|X_s|^p)^{\frac{2}{p}}\big)\d s\Big)^{\frac{p}{2}}\nonumber\\&\leq C_{p,T}\big(1+E|X_0|^p\big)\Big(\int_{0}^{t}F_2(t',t,s)\d s\Big)^{\frac{p}{2}}\nonumber\\&\leq C_{p,T}\big(1+E|X_0|^p\big)|t'-t|^{\gamma p}.\nonumber
	\end{align}
	Hence, for all $0\leq t<t'\leq T$,
	$$E|X_{t'}-X_t|^p\leq C_{p,T}\big(1+E|X_0|^p\big)|t'-t|^{\gamma p}.$$
	Existence of a continuous version as well as the bouud now follows from the Kolmogorov's continuity theorem, see Theorem I. $2.1$ in \cite{RY}.
\end{proof}
\section{Propagation of chaos}\label{Propagation of chaos}
In this section, let us now introduce the interacting particle system of \eqref{Volterra eq} in order to study propagation of chaos property (see, e.g., \cite{SS}). More precisely, we consider the following interacting particle system with singular kernels:
\begin{align}\label{interacting particle system}
	X_t^{i,N}=X_0^i+\int_{0}^{t}b(t,s, X_s^{i,N}, \mu^{X,N}_s)\d s+\int_{0}^{t}\sigma(t,s, X_s^{i,N}, \mu^{X,N}_s)\d W_s,
\end{align}
almost surely, where the empirical measures for any $t\in [0,T]$ is defined by
$$\mu^{X,N}_t(\cdot):=\frac{1}{N}\sum_{i=1}^{N}\delta_{X_t^{i,N}}(\cdot),$$
and $\delta_x$ denotes the Dirac measure at point $x$. Also, we consider the following Volterra system of non-interacting particles,
\begin{align}\label{limnit of interacting particle system}
	X_t^{i}=X_0^i+\int_{0}^{t}b(t,s, X_s^{i}, \mathcal{L}_{X_s^{i}})\d s+\int_{0}^{t}\sigma(t,s, X_s^{i}, \mathcal{L}_{X_s^{i}})\d W_s,
\end{align}
almost surely for any $t\in [0,T]$ and $i\in\{1,\cdots, N\}$.
\begin{lem}\label{estimate of X_t^{i,q}}
	Assume that $X_0\in L^{q}(\Omega,\mathcal{F}_0,P)$ with some $q>p$ and assumptions $\mathbf{(H1)}$-$\mathbf{(H3)}$ hold, then there exists a constant depending only on $q,T, X_0^i$ such that for all $i\in\{1,\cdots, N\}$,
	$$E|X^i_t|^q\leq C_{q,T}\big(E|X_0|^q+1\big).$$
\end{lem}
\begin{proof}
	Applying BDG's inequality, H$\mathrm{\ddot{o}}$lder's inequality, $\mathbf{(H1)}$ and $\mathbf{(H3)}$ yields
	\begin{align}
		E|X_t^i|^q&=E\Big|X_0^i+\int_{0}^{t}b(t,s, X_s^{i}, \mathcal{L}_{X_s^{i}})\d s+\int_{0}^{t}\sigma(t,s, X_s^{i}, \mathcal{L}_{X_s^{i}})\d W_s\Big|^q\nonumber\\&\leq 3^{q-1}E|X_0^i|^q+3^{q-1}E\Big|\int_{0}^{t}b(t,s, X_s^{i}, \mathcal{L}_{X_s^{i}})\d s\Big|^q+3^{q-1}E\Big|\int_{0}^{t}\sigma(t,s, X_s^{i}, \mathcal{L}_{X_s^{i}})\d W_s\Big|^q\nonumber\\&\leq3^{q-1}E|X_0^i|^q+3^{q-1}E\Big|\int_{0}^{t}K_3(t,s)\cdot(1+|X_s^i|+\mathbb{W}_2(\mathcal{L}_{X_s^{i}},\delta_0))\d s\Big|^q+C_qE\Big|\int_{0}^{t}\|\sigma(t,s, X_s^{i}, \mathcal{L}_{X_s^{i}})\|^2\d s\Big|^\frac{q}{2}\nonumber\\&\leq 3^{q-1}E|X^i_0|^q +C_q\int_{0}^{t}K_3(t,s)\cdot E\Big(1+|X^{i}_s|^p+\mathbb{W}^q_{2}(\mathcal{L}_{X^{i}_s},\delta_0)\Big)\d s\cdot\Big(\int_{0}^{t}K_3(t,s)\d s\Big)^{q-1}\nonumber\\&+ C_q\int_{0}^{t}K_4(t,s)\cdot E\Big(1+|X^{i}_s|^p+\mathbb{W}^q_{2}(\mathcal{L}_{X^{i}_s},\delta_0)\Big)\d s\cdot\Big(\int_{0}^{t}K_3(t,s)\d s\Big)^{\frac{q-1}{2}}\nonumber\\&\leq 3^{q-1}E|X_0|^p+C_{p,T}\int_{0}^{t}K_3(t,s)\cdot E(1+|X_s^i|^q)\d s+C_{q,T}\int_{0}^{t}K_4(t,s)\cdot E(1+|X_s^i|^q)\d s\nonumber\\&\leq 3^{q-1}E|X_0|^p+C_{q,T}+C_{q,T}\int_{0}^{t}\Big(K_3(t,s)+K_4(t,s)\Big)\cdot E|X_s^i|^q\d s\nonumber\\&\leq C_{q,T}\big(E|X_0|^q+1\big)+C_{q,T}\int_{0}^{t}\Big(K_3(t,s)+K_4(t,s)\Big)\cdot E|X_s^i|^q \d s,\nonumber
	\end{align}
	then by Lemma \ref{Gronwall}, we have that
	\begin{align}
		E|X^i_t|^q\leq C_{q,T}\big(E|X_0|^q+1\big).\nonumber
	\end{align}
\end{proof}
\begin{thm}
	Assume that $X_0\in L^{q}(\Omega,\mathcal{F}_0,P)$ with some $q>p$ and assumptions $\mathbf{(H1)}$-$\mathbf{(H3)}$ hold, then there exists a constant C depending only on $d, p,q,T,X_0^i$ such that for all $N\geq 1$,
	$$\sup_{i=1,\cdots,N}\sup_{t\in [0,T]}E|X_t^{i,N}-X_t^i|^p\leq C_{p,d,q,T,X_0^i}\times \begin{cases}
		N^{-\frac{1}{2}}+N^{-\frac{p-q}{q}} &\quad\text{if}\quad p>\frac{d}{2}\ \text{and}\ q\neq 2p,\\
		N^{-\frac{1}{2}}\log(1+N)+N^{-\frac{p-q}{q}} &\quad\text{if}\quad p=\frac{d}{2}\ \text{and}\ q\neq 2p,\\
		N^{-\frac{p}{d}}+N^{-\frac{p-q}{q}} &\quad\text{if}\quad p\in(0,\frac{d}{2})\ \text{and}\ q\neq\frac{d}{d-p},
	\end{cases}$$
\end{thm}
\begin{proof}
	By BDG's inequality, H$\mathrm{\ddot{o}}$lder's inequality, $\mathbf{(H1)}$ and $\mathbf{(H2)}$, we obtain
	\begin{align}
		&E|X_t^{i,N}-X_t^i|^p\nonumber\\&=E\big|\int_{0}^{t}\big(b(t,s, X_s^{i,N}, \mu^{X,N}_s)-b(t,s, X_s^{i}, \mathcal{L}_{X_s^{i}})\big)\d s+\int_{0}^{t}\big(\sigma(t,s, X_s^{i,N}, \mu^{X,N}_s)-\sigma(t,s, X_s^{i}, \mathcal{L}_{X_s^{i}})\big)\d W_s\big|^p\nonumber\\&\leq 2^{p-1}E\Big(\Big|\int_{0}^{t}|b(t,s,X^{i,N}_s,\mu^{X,N}_s)-b(t,s,X^{i}_s,\mathcal{L}_{X^{i}_s})|\d s\Big|^p\Big)\nonumber\\&+2^{p-1}E\Big(\Big|\int_{0}^{t}\big(\sigma(t,s,X^{i,N}_s,\mu^{X,N}_s)-\sigma(t,s,X^{i}_s,\mathcal{L}_{X^{i}_s})\big)\d W_s\Big|^p\Big)\nonumber\\&\leq 2^{p-1}E\Big(\int_{0}^{t}K_1(t,s)\big(\big|X^{i,N}_s-X^{i}_s|+\mathbb{W}_{2}(\mu^{X,N}_s,\mathcal{L}_{X^{i}_s})\big)\d s\Big)^p\nonumber\\&+C_p E\Big(\int_{0}^{t}\big\|\sigma(t,s,X^{i,N}_s,\mu^{X,N}_s)-\sigma(t,s,X^{i}_s,\mathcal{L}_{X^{i}_s})\big\|^2\d s\Big)^{\frac{p}{2}}\nonumber\\&\leq  C_p\int_{0}^{t}K_1(t,s)\cdot E\Big(\big|X^{i,N}_s-X^{i}_s\big|^p+\mathbb{W}^p_{2}(\mu^{X,N}_s,\mathcal{L}_{X^{i}_s})\Big)\d s\cdot\Big(\int_{0}^{t}K_1(t,s)\d s\Big)^{p-1}\nonumber\\&+ C_p\int_{0}^{t}K_2(t,s)\cdot E\Big(\big|X^{i,N}_s-X^{i}_s\big|^p+\mathbb{W}^p_{2}(\mu^{X,N}_s,\mathcal{L}_{X^{i}_s})\Big)\d s\cdot\Big(\int_{0}^{t}K_2(t,s)\d s\Big)^{\frac{p-1}{2}}\nonumber\\&\leq C_{p,T}\int_{0}^{t}K_1(t,s)\cdot E|X_s^{i,N}-X_s^i|^p\d s+C_{p,T}\int_{0}^{t}K_2(t,s)\cdot E|X_s^{i,N}-X_s^i|^p\d s\nonumber\\&+C_{p,T}\int_{0}^{t}\Big(K_1(t,s)+K_2(t,s)\Big)\cdot E\mathbb{W}^p_{2}(\mu^{X,N}_s,\mathcal{L}_{X^{i}_s})\d s.\nonumber
	\end{align}
	Note that
	\begin{align}
		\mathbb{W}^p_{2}(\mu^{X,N}_s,\mathcal{L}_{X^{i}_s})&\leq \Big(\mathbb{W}^2_{2}(\mu^{X,N}_s,\mathcal{L}_{X^{i}_s})\Big)^{\frac{p}{2}}\leq \Big(\frac{2}{N}\sum_{i=1}^{N}|X_s^i-X_s^{i,N}|^2+2\mathbb{W}^2_{2}(\frac{1}{N}\sum_{i=1}^{N}\delta_{X_t^{i}},\mathcal{L}_{X^{i}_s})\Big)^{\frac{p}{2}}\nonumber\\&\leq C_p \Big(\frac{1}{N}\sum_{i=1}^{N}|X_s^i-X_s^{i,N}|^2\Big)^{\frac{p}{2}}+C_p\mathbb{W}^p_{2}(\frac{1}{N}\sum_{i=1}^{N}\delta_{X_t^{i}},\mathcal{L}_{X^{i}_s}).\nonumber
	\end{align}
	Therefore, we have that
	\begin{align}
		&E\mathbb{W}^p_{2}(\mu^{X,N}_s,\mathcal{L}_{X^{i}_s})\nonumber\\&\leq C_p E\Big(\frac{1}{N}\sum_{i=1}^{N}|X_s^i-X_s^{i,N}|^2\Big)^{\frac{p}{2}}+C_p E\mathbb{W}^p_{2}(\frac{1}{N}\sum_{i=1}^{N}\delta_{X_t^{i}},\mathcal{L}_{X^{i}_s})\nonumber.
	\end{align}
	As a consequence, by extended Minkowski's inequality and H$\mathrm{\ddot{o}}$lder's inequality we have
	\begin{align}
		&E|X_t^{i,N}-X_t^i|^p\nonumber\\&\leq C_{p,T}\int_{0}^{t}K_1(t,s)\cdot E|X_s^{i,N}-X_s^i|^p\d s+C_{p,T}\int_{0}^{t}K_2(t,s)\cdot E|X_s^{i,N}-X_s^i|^p\d s\nonumber\\&+C_{p,T}\int_{0}^{t}\Big(K_1(t,s)+K_2(t,s)\Big)\cdot E\mathbb{W}^p_{2}(\mu^{X,N}_s,\mathcal{L}_{X^{i}_s})\d s\nonumber\\&\leq C_{p,T}\int_{0}^{t}K_1(t,s)\cdot E|X_s^{i,N}-X_s^i|^p\d s+C_{p,T}\int_{0}^{t}K_2(t,s)\cdot E|X_s^{i,N}-X_s^i|^p\d s\nonumber\\&+C_{p,T}\int_{0}^{t}\Big(K_1(t,s)+K_2(t,s)\Big)\cdot \Big[E\Big(\frac{1}{N}\sum_{i=1}^{N}|X_s^i-X_s^{i,N}|^2\Big)^{\frac{p}{2}}+ E\mathbb{W}^p_{2}(\frac{1}{N}\sum_{i=1}^{N}\delta_{X_t^{i}},\mathcal{L}_{X^{i}_s})\Big]\d s\nonumber\\&\leq C_{p,T}\int_{0}^{t}K_1(t,s)\cdot E|X_s^{i,N}-X_s^i|^p\d s+C_{p,T}\int_{0}^{t}K_2(t,s)\cdot E|X_s^{i,N}-X_s^i|^p\d s\nonumber\\&+C_{p,T}\int_{0}^{t}\Big(K_1(t,s)+K_2(t,s)\Big)\cdot \Big[\Big(\frac{1}{N}\sum_{i=1}^{N}\big(E|X_s^i-X_s^{i,N}|^p\big)^{\frac{2}{p}}\Big)^{\frac{p}{2}}+ E\mathbb{W}^p_{2}(\frac{1}{N}\sum_{i=1}^{N}\delta_{X_t^{i}},\mathcal{L}_{X^{i}_s})\Big]\d s\nonumber\\&\leq C_{p,T}\int_{0}^{t}\Big(K_1(t,s)+K_2(t,s)\Big)\cdot E|X_s^{i,N}-X_s^i|^p\d s+C_{p,T}\int_{0}^{t}\Big(K_1(t,s)+K_2(t,s)\Big)\cdot\nonumber E\mathbb{W}^p_{p}(\frac{1}{N}\sum_{i=1}^{N}\delta_{X_t^{i}},\mathcal{L}_{X^{i}_s})\d s,
	\end{align}
	then applying Theorem $1$ of \cite{FG}, we obtain
	\begin{align}
		E\mathbb{W}^p_{p}(\frac{1}{N}\sum_{i=1}^{N}\delta_{X_t^{i}},\mathcal{L}_{X^{i}_s})\leq C_{p,d,q}M_q^{\frac{q}{p}}(\mathcal{L}_{X_s^i})\times \begin{cases}
			N^{-\frac{1}{2}}+N^{-\frac{p-q}{q}} &\quad\text{if}\quad p>\frac{d}{2}\ \text{and}\ q\neq 2p,\\
			N^{-\frac{1}{2}}\log(1+N)+N^{-\frac{p-q}{q}} &\quad\text{if}\quad p=\frac{d}{2}\ \text{and}\ q\neq 2p,\\
			N^{-\frac{p}{d}}+N^{-\frac{p-q}{q}} &\quad\text{if}\quad p\in(0,\frac{d}{2})\ \text{and}\ q\neq\frac{d}{d-p},
		\end{cases}\nonumber
	\end{align}
	where $M_q(\mathcal{L}_{X_s^i})$ is defined by
	$$M_q(\mathcal{L}_{X_s^i})=\int_{\mathbb{R}^d}|X_s^i|^q\mathcal{L}_{X_s^i}(\d X_s^i).$$
	By Lemma \ref{estimate of X_t^{i,q}}, $M_q(\mathcal{L}_{X_s^i})<\infty$, using Lemma \ref{Gronwall}, we can get desired result.
\end{proof}
\section{Euler schemes for Volterra type McKean-Vlasov SDEs}\label{Euler schemes for Volterra type McKean-Vlasov SDEs}
\subsection{Discretized scheme for interacting particle system}
In this subsection, we construct a Euler schemes for the interacting particle system \eqref{interacting particle system} associated with \eqref{limnit of interacting particle system}. When the coefficients $b$ and $\sigma$ are independent of the law of $X_t$, Zhang \cite{ZX1} proposed Euler schemes for classsical stochastic Volterra equations with singular kernels, where the author made a delicate design to avoid touching the singular points. We now propose the Euler schemes for the interacting particle system \eqref{interacting particle system} associated with \eqref{limnit of interacting particle system}. Without loss of generality, we assume $T=1$.
\begin{align}\label{Euler}
	X_t^{i,N,n}=X_0^i+\int_{0}^{t}b^{i,N,n}(\tilde{t},\tilde{s}_n,X_{s_n}^{i,N,n},\mu^{X,N,n}_{s_n})\d s+\int_{0}^{t}\sigma^{i,N,n}(\tilde{t},\tilde{s}_n,X_{s_n}^{i,N,n},\mu^{X,N,n}_{s_n})\d W_s, \ t\in [0,1],
\end{align}
for each $i\in\{1,\cdots,N\}$, where $X_0^i:=\xi \in\mathbb{R}^d, s_n:=\frac{[2^n s]}{2^n}$ and
$$\mu^{X,N,n}_t(\cdot):=\frac{1}{N}\sum_{i=1}^{N}\delta_{X_t^{i,N,n}}(\cdot),$$
almost surely for $t\in [0,1]$ and $n\in\mathbb{N}$. Here $[a]$ denotes the integral part of a real number $a$. And $0<\tilde{s}_n<\tilde{t}$ are defined by
\begin{align}\label{def of t}
	\tilde{t}:=t\cdot \mathbb{I}_{\{t\geq 2^{-n}\}}+2^{-n}\cdot\mathbb{I}_{\{t\in [0,2^{-n}]\}},
\end{align}
\begin{align}\label{def of s_n}
	\tilde{s}_n:=s_n\cdot \mathbb{I}_{\{s\in [2^{-n},1)\}}+2^{-n-1}\cdot\mathbb{I}_{\{s\in (0,2^{-n})\}}.
\end{align}
In order to index the dependence of $\xi$, we denote the solution of \eqref{limnit of interacting particle system} and \eqref{Euler} by $X^{i}_{t,\xi}$ and $X^{i,N,n}_{t,\xi}$. Moreover, we also need to add more assumptions below.

$\mathbf{(H1'')}$ High-order integrability condition:
$$ \sup_{t\in [0,1]}\int_{0}^{t}\Big[K^{\a}_1(t,s)+K^{\a}_2(t,s)+K^{\a}_3(t,s)+K^{\a}_4(t,s)\Big]\d s\leq C, \quad t\in [0,1],$$
$$\sup_{t\in [0,1]}\int_{0}^{t}\Big[K^{\a}_1(\tilde{t},\tilde{s}_n)+K^{\a}_2(\tilde{t},\tilde{s}_n)+K^{\a}_3(\tilde{t},\tilde{s}_n)+K^{\a}_4(\tilde{t},\tilde{s}_n)\Big]\d s\leq C, \quad t\in [0,1],$$
for some $\a>1$, where $K_i, i=1,2,3,4$ are from $\mathbf{(H2)}$ and $\mathbf{(H3)}$.

$\mathbf{(H4')}$ For all $t,t',s\in [0,1]$, $\mu\in \mathcal{P}_2(\mathbb{R}^d)$ and $x\in \mathbb{R}^{d}$,
\begin{align}
	|b(t',s,x,\mu)-b(t,s,x,\mu)|&\leq F_1(t',t,s)\Big(1+|x|+\mathbb{W}_2^2(\mu,\delta_0)\Big),\nonumber\\||\sigma(t',s,x,\mu)-\sigma(t,s,x,\mu)||^2&\leq F_2(t',t,s)\Big(1+|x|^2+\mathbb{W}_2^2(\mu,\delta_0)\Big),\nonumber
\end{align}
where $F_i(t',t,s), i=1,2,$ are positive function on $[0,1]\times[0,1]\times[0,1]$, and satisfy for some $\gamma>0$
$$\int_{0}^{t\wedge t'}\big(F_1(t',t,s)+F_2(t',t,s)\big)\d s\leq C|t-t'|^{\gamma},$$
and
$$\int_{0}^{t\wedge t'}\big(F_1(\tilde{t'},\tilde{t},\tilde{s}_n)+F_2(\tilde{t'},\tilde{t},\tilde{s}_n)\big)\d s\leq C|\tilde{t}-\tilde{t'}|^{\gamma}.$$

$\mathbf{(H5)}$ For all $t,s,s'\in [0,1]$, $\mu\in \mathcal{P}_2(\mathbb{R}^d)$ and $x\in \mathbb{R}^{d}$,
\begin{align}
	|b(t,s,x,\mu)-b(t,s',x,\mu)|&\leq F_3(t,s,s')\Big(1+|x|+\mathbb{W}_2^2(\mu,\delta_0)\Big),\nonumber\\||\sigma(t,s,x,\mu)-\sigma(t,s',x,\mu)||^2&\leq F_4(t,s,s')\Big(1+|x|^2+\mathbb{W}_2^2(\mu,\delta_0)\Big),\nonumber
\end{align}
where $F_3(t,s,s')$ and $F_4(t,s,s')$ satisfy for some $\delta>0$ and all $t\in [0,1]$
$$\int_{0}^{t}\big(F_3(\tilde{t},s,\tilde{s}_n)+F_4(\tilde{t},s,\tilde{s}_n)\big)\d s\leq C 2^{-\delta n}, \ n\in\mathbb{N}.$$
\subsection{The Convergence Rate for the Euler schemes}
We first prove two lemmas.
\begin{lem}\label{estimate 1}
	For any $p\geq\frac{2\alpha}{\alpha-1}$, suppose that $\mathbf{(H1'')}$ and $\mathbf{(H2)}$-$\mathbf{(H3)}$ hold, for all $n,N\in\mathbb{N}$, $t\in [0,1]$ and $\xi,\eta\in\mathbb{R}^{d}$, we have
	$$\sup_{i=1,\cdots,N}\sup_{t\in [0,1]}E|X_{t,\xi}^{i,N,n}|^p\leq C_p(1+|\xi|^p),$$
	and
	$$\sup_{i=1,\cdots,N}\sup_{t\in [0,1]}E|X_{t,\xi}^{i,N,n}-X_{t,\eta}^{i,N,n}|^p\leq C_p|\xi-\eta|^p,$$
	where $C_p>0$ does not depend on $n, N\in\mathbb{N}$.
\end{lem}
\begin{proof}
	Recall the discretized scheme \eqref{Euler},  by H$\mathrm{\ddot{o}}$lder's inequality, BDG's inequality, $\mathbf{(H1'')}$ and $\mathbf{(H3)}$, we have for $p\geq\frac{2\alpha}{\alpha-1}$
	\begin{align}
		&E|X_{t,\xi}^{i,N,n}|^p=E\Bigg|\xi+\int_{0}^{t}b^{i,N,n}(\tilde{t},\tilde{s}_n,X_{s_n,\xi}^{i,N,n},\mu^{X,N,n}_{s_n,\xi})\d s+\int_{0}^{t}\sigma^{i,N,n}(\tilde{t},\tilde{s}_n,X_{s_n,\xi}^{i,N,n},\mu^{X,N,n}_{s_n,\xi})\d W_s\Bigg|^p\nonumber\\&\leq C_p|\xi|^p+C_p E\Big|\int_{0}^{t}b^{i,N,n}(\tilde{t},\tilde{s}_n,X_{s_n,\xi}^{i,N,n},\mu^{X,N,n}_{s_n,\xi})\d s\Big|^p+C_p E\Big|\int_{0}^{t}\sigma^{i,N,n}(\tilde{t},\tilde{s}_n,X_{s_n,\xi}^{i,N,n},\mu^{X,N,n}_{s_n,\xi})\d W_s\Big|^p\nonumber\\&\leq C_p|\xi|^p+C_pE\Big|\int_{0}^{t}K_3(\tilde{t},\tilde{s}_n)\cdot\Big(1+X_{s_n,\xi}^{i,N,n}+\mathbb{W}_2(\mu^{X,N,n}_{s_n,\xi},\delta_0)\Big)\d s\Big|^p\nonumber\\&+C_p E\Big(\int_{0}^{t}\big\|\sigma^{i,N,n}(\tilde{t},\tilde{s}_n,X_{s_n,\xi}^{i,N,n},\mu^{X,N,n}_{s_n,\xi})\big\|^2\d s\Big)^{\frac{p}{2}}\nonumber\\&\leq\nonumber C_p|\xi|^p+C_pE\Bigg(\int_{0}^{t} \Big(1+|X^{i,N,n}_{s_n,\xi}|+\mathbb{W}_{2}(\mathcal{L}_{X^{i,N,n}_{s_n,\xi}},\delta_0)\Big)^{\a^*}ds\Bigg)^{\frac{p}{\a^*}}\cdot\Big(\int_{0}^{t}K^{\a}_3(\tilde{t},\tilde{s}_n)ds\Big)^{\frac{p}{\a}}\nonumber\\&+C_p E\Bigg(\int_{0}^{t} \Big(1+|X^{i,N,n}_{s_n,\xi}|^2+\mathbb{W}^2_{2}(\mathcal{L}_{X^{i,N,n}_{s_n,\xi}},\delta_0)\Big)^{\a^*}\d s\Bigg)^{\frac{p}{2\a^*}}\cdot\Big(\int_{0}^{t}K^{\a}_4(\tilde{t},\tilde{s})\d s\Big)^{\frac{p}{2\a}}\nonumber\\&\leq\nonumber C_p|\xi|^p+C_pE\Bigg(\int_{0}^{t} \Big(1+|X^{i,N,n}_{s_n,\xi}|+\mathbb{W}_{2}(\mathcal{L}_{X^{i,N,n}_{s_n,\xi}},\delta_0)\Big)^{\a^*}\d s\Bigg)^{\frac{p}{\a^*}}\nonumber\\&+C_p E\Bigg(\int_{0}^{t} \Big(1+|X^{i,N,n}_{s_n,\xi}|^2+\mathbb{W}^2_{2}(\mathcal{L}_{X^{i,N,n}_{s_n,\xi}},\delta_0)\Big)^{\a^*}\d s\Bigg)^{\frac{p}{2\a^*}}\nonumber\\&\leq C_p|\xi|^p+C_p\int_{0}^{t}E\Big(1+|X^{i,N,n}_{s_n,\xi}|+\mathbb{W}_{2}(\mathcal{L}_{X^{i,N,n}_{s_n,\xi}},\delta_0)\Big)^{p}\d s+C_p\int_{0}^{t}E\Big(1+|X^{i,N,n}_{s_n,\xi}|^2+\mathbb{W}^2_{2}(\mathcal{L}_{X^{i,N,n}_{s_n,\xi}},\delta_0)\Big)^{\frac{p}{2}}\d s\nonumber\\&\leq C_p(1+|\xi|^p)+C_p \int_{0}^{t}\Big(E|X^{i,N,n}_{s_n,\xi}|^p+E\mathbb{W}^p_{2}(\mathcal{L}_{X^{i,N,n}_{s_n,\xi}},\delta_0)\Big)\d s\nonumber,
	\end{align}
	for $\alpha>1, \alpha^*:=\frac{p}{p-1}$, any $t\in [0,1], i\in \{1,\cdots, N\}$ and $n,N \in\mathbb{N}$. By a simple calculation (see e.g., Lemma $2.3$ in \cite{dST}), one can observe that
	$$\mathbb{W}^2_{2}(\mathcal{L}_{X^{i,N,n}_{s_n,\xi}},\delta_0)=\frac{1}{N}\sum_{i=1}^{N}|X_{s_n,\xi}^{i,N,n}|^2,$$
	which yields that
	\begin{align}
		&\sup_{i=1,\cdots,N}\sup_{s\in [0,t]}E|X_{s,\xi}^{i,N,n}|^p\nonumber\\&\leq C_p(1+|\xi|^p)+C_p\int_{0}^{t}\sup_{i=1,\cdots,N}\sup_{u\in [0,s]}E|X_{u,\xi}^{i,N,n}|^p\d s.\nonumber
	\end{align}
	Applying Gronwall's inequality, we have
	$$\sup_{i=1,\cdots,N}\sup_{t\in [0,1]}E|X_{t,\xi}^{i,N,n}|^p\leq C_p(1+|\xi|^p).$$
	The second estimate can be obtained
	by similar estimations.
\end{proof}
\begin{rem}
	Note that classical Gronwall's inequality instead of Volterra type Gronwall's equality is used to prove above Lemma. High-order integrability condition is required. We have attempted to apply Volterra type Gronwall's equality, but it seems that Volterra type Gronwall's equality is invaild when we obtain the following estimate due to the different indexes $t,\tilde{t}:$
	$$E|X_{t,\xi}^{i,N,n}|^p\leq C_p(1+|\xi|^p)+C_p \int_{0}^{t}\Big(K_3(\tilde{t},s_n)+K_4(\tilde{t},s_n)\Big)\cdot E|X_{s_n,\xi}^{i,N,n}|^p\d s.$$
\end{rem}
We now prove the following estimate associated with $X^{i,N,n}_{t',\xi}-X^{i,N,n}_{t,\xi}$.
\begin{lem}\label{estimate of t}
	For any $p\geq2$ sufficiently large,  suppose that $\mathbf{(H1'')}$, $\mathbf{(H2)}$, $\mathbf{(H3)}$ and $\mathbf{(H4')}$ hold. Then there exists a constant $C_p>0$ such that for all $n,N\in\mathbb{N}$, $t,t'\in [0,1]$ and $\xi\in\mathbb{R}^{d}$, we have
	$$E|X^{i,N,n}_{t',\xi}-X^{i,N,n}_{t,\xi}|\leq C_p(1+|\xi|^p)|t'-t|^{\theta p},$$
	where $\theta>0$ depends only $\a,\gamma$.
\end{lem}
\begin{proof}
	For any $t<t'$, we first divide $X^{i,N,n}_{t',\xi}-X^{i,N,n}_{t,\xi}$ into four parts.
	\begin{align}
		&X^{i,N,n}_{t',\xi}-X^{i,N,n}_{t,\xi}\nonumber\\&=\int_{t}^{t'}b^{i,N,n}(\tilde{t'},\tilde{s}_n,X_{s_n,\xi}^{i,N,n},\mu^{X,N,n}_{s_n,\xi})\d s+\int_{t}^{t'}\sigma^{i,N,n}(\tilde{t'},\tilde{s}_n,X_{s_n,\xi}^{i,N,n},\mu^{X,N,n}_{s_n,\xi})\d W_s\nonumber\\&+\int_{0}^{t}\Big(b^{i,N,n}(\tilde{t'},\tilde{s}_n,X_{s_n,\xi}^{i,N,n},\mu^{X,N,n}_{s_n,\xi})-b^{i,N,n}(\tilde{t},\tilde{s}_n,X_{s_n,\xi}^{i,N,n},\mu^{X,N,n}_{s_n,\xi})\Big)\d s\nonumber\\&+\int_{0}^{t}\Big(\sigma^{i,N,n}(\tilde{t'},\tilde{s}_n,X_{s_n,\xi}^{i,N,n},\mu^{X,N,n}_{s_n,\xi})-\sigma^{i,N,n}(\tilde{t},\tilde{s}_n,X_{s_n,\xi}^{i,N,n},\mu^{X,N,n}_{s_n,\xi})\Big)\d W_s\nonumber\\&:=\Xi_1+\Xi_2+\Xi_3+\Xi_4.\nonumber
	\end{align}
	For the term $\Xi_2$, by Lemma \ref{estimate 1}, H$\mathrm{\ddot{o}}$lder's inequality, BDG's inequality, $\mathbf{(H1'')}$ and $\mathbf{(H3)}$, we have
	\begin{align}
		E|\Xi_2|^p&=E\Big|\int_{t}^{t'}\sigma^{i,N,n}(\tilde{t'},\tilde{s}_n,X_{s_n,\xi}^{i,N,n},\mu^{X,N,n}_{s_n,\xi})\d W_s\Big|^p\nonumber\\&\leq C_p E\Big(\int_{t}^{t'}\big\|\sigma^{i,N,n}(\tilde{t'},\tilde{s}_n,X_{s_n,\xi}^{i,N,n},\mu^{X,N,n}_{s_n,\xi})\big\|^2\d s\Big)^{\frac{p}{2}}\nonumber\\&\leq C_pE\Big(\int_{t}^{t'}K_4(\tilde{t'},\tilde{s}_n,\xi)\cdot\big(1+|X_{s_n,\xi}^{i,N,n}|^2+\mathbb{W}^2_2(\mu^{X,N,n}_{s_n,\xi},\delta_0)\big)\d s\Big)^{\frac{p}{2}}\nonumber\\&\leq C_p\Big(\int_{t}^{t'}K^{\a}_4(\tilde{t'},\tilde{s}_n)\d s\Big)^{\frac{p}{2\a}}E\Big(\int_{t}^{t'}\big(1+|X_{s_n,\xi}^{i,N,n}|^2+\mathbb{W}^2_2(\mu^{X,N,n}_{s_n,\xi},\delta_0)\big)^{\a^*}\d s\Big)^{\frac{p}{2\a^*}}\nonumber\\&\leq C_p(1+|\xi|^p)|t'-t|^{\frac{p}{2\a^*}}.\nonumber
	\end{align}
	For the term $\Xi_4$, combining BDG's inequality, extended Minkowski's inequality, $\mathbf{(H1'')}$, $\mathbf{(H4)}$ with Lemma \ref{estimate 1}, we have that
	\begin{align}
		&E|\Xi_4|^p\nonumber\\&=E\Big|\int_{0}^{t}\Big(\sigma^{i,N,n}(\tilde{t'},\tilde{s}_n,X_{s_n,\xi}^{i,N,n},\mu^{X,N,n}_{s_n,\xi})-\sigma^{i,N,n}(\tilde{t},\tilde{s}_n,X_{s_n,\xi}^{i,N,n},\mu^{X,N,n}_{s_n})\Big)\d W_s\Big|^p\nonumber\\&\leq C_p E\Big(\int_{0}^{t}\big\|\sigma^{i,N,n}(\tilde{t'},\tilde{s}_n,X_{s_n,\xi}^{i,N,n},\mu^{X,N,n}_{s_n,\xi})-\sigma^{i,N,n}(\tilde{t},\tilde{s}_n,X_{s_n,\xi}^{i,N,n},\mu^{X,N,n}_{s_n,\xi})\big\|^2\d s\Big)^{\frac{p}{2}}\nonumber\\&\leq C_pE\Big(\int_{0}^{t}F_2(\tilde{t'},\tilde{t},\tilde{s}_n)\cdot\big(1+|X_{s_n,\xi}^{i,N,n}|^2+\mathbb{W}^2_2(\mu^{X,N,n}_{s_n,\xi},\delta_0)\big)\d s\Big)^{\frac{p}{2}}\nonumber\\&\leq C_p\Bigg(\int_{0}^{t}F_2(\tilde{t'},\tilde{t},\tilde{s}_n)\cdot \big(E\big[1+|X^{i,N,n}_{s_n,\xi}|^p+\mathbb{W}_2^p(\mathcal{L}_{X^{i,N,n}_{s_n,\xi}},\delta_0)\big]\big)^{\frac{2}{p}} \d s\Bigg)^{\frac{p}{2}}\nonumber\\&\leq C_p(1+|\xi|^p)\Big(\int_{0}^{t}F_2(\tilde{t'},\tilde{t},\tilde{s}_n)\d s\Big)^{\frac{p}{2}}\leq C_p(1+|\xi|^p)|t'-t|^{\frac{\gamma p}{2}}\nonumber.
	\end{align}
	Similarly, we may deal with the terms $\Xi_1$ and $\Xi_3$, and obtain
	\begin{align}
		&E|\Xi_1|^p\leq C_p (1+|\xi|^p)|t'-t|^{\frac{p}{\a^*}}\nonumber\\&E|\Xi_3|^p\leq C_p (1+|\xi|^p)|t'-t|^{\gamma p}.\nonumber
	\end{align}
	Combining the above calculation, we have
	\begin{align}
		E|X^{i,N,n}_{t'}-X^{i,N,n}_t|\leq C_p(1+|X_0^i|^p)\Big(|t'-t|^{\frac{p}{\a^*}}+|t'-t|^{\frac{p}{2\a^*}}+|t'-t|^{\gamma p}+|t'-t|^{\frac{\gamma p}{2}}\Big),\nonumber
	\end{align}
	which concludes the proof of Lemma \ref{estimate of t}.
\end{proof}
The main results of this section are given in the following Theorems.
\begin{thm}\label{rate}
	Assume that $\mathbf{(H1'')}$, $\mathbf{(H2)}$, $\mathbf{(H3)}$ ,$\mathbf{(H4')}$ and $\mathbf{(H5)}$ hold, the Euler scheme \eqref{Euler} converges to the true solution of the interacting particle system \eqref{interacting particle system} associated with \eqref{limnit of interacting particle system} in a strong sense with $L^p$, i.e.,
	$$\sup_{i=1,\cdots,N}\sup_{t\in [0,1]}E|X_{t,\xi}^{i,N}-X_{t,\xi}^{i,N,n}|^p\leq C_p(1+|\xi|^p)\cdot 2^{-n\eta p},$$
	for $p$ sufficiently large, where the constant $C_p$ does depend on $n,N\in\mathbb{N}$ and $\eta>0$ depends only $\a,\gamma,\delta$.
\end{thm}
\begin{proof}
	Set
	$$Z^{i,N,n}_{t,\xi}:=X^{i,N}_{t,\xi}-X^{i,N,n}_{t,\xi}.$$
	Then for any $t\in [0,1], i\in\{1,\cdots,N\}$ and $n,N\in\mathbb{N}$, (where $b(t,s,X_{s,\xi}^{i,N,n},\mu^{X,N,n}_{s,\xi})=b^{i,N,n}(t,s,X_{s,\xi}^{i,N,n},\mu^{X,N,n}_{s,\xi})$ and $ \sigma(t,s,X_{s,\xi}^{i,N,n},\mu^{X,N,n}_{s,\xi})=\sigma^{i,N,n}(t,s,X_{s,\xi}^{i,N,n},\mu^{X,N,n}_{s,\xi}$)
	\begin{align}
		Z^{i,N,n}_{t,\xi}&=\int_{0}^{t}\big(b(t,s, X_{s,\xi}^{i,N}, \mu^{X,N}_{s,\xi})-b^{i,N,n}(\tilde{t},\tilde{s}_n,X_{s_n,\xi}^{i,N,n},\mu^{X,N,n}_{s_n,\xi})\big)\d s\nonumber\\&+\int_{0}^{t}\big(\sigma(t,s, X_{s,\xi}^{i,N}, \mu^{X,N}_{s,\xi})-\sigma^{i,N,n}(\tilde{t},\tilde{s}_n,X_{s_n,\xi}^{i,N,n},\mu^{X,N,n}_{s_n,\xi})\big)\d W_s\nonumber\\&=\int_{0}^{t}\Big(b(t,s, X_{s,\xi}^{i,N}, \mu^{X,N}_{s,\xi})-b(t,s,X_{s,\xi}^{i,N,n},\mu^{X,N,n}_{s,\xi})\Big)\d s\nonumber\\&+\int_{0}^{t}\Big(b^{i,N,n}(t,s, X_{s,\xi}^{i,N,n}, \mu^{X,N,n}_{s,\xi})-b^{i,N,n}(t,s,X_{s_n,\xi}^{i,N,n},\mu^{X,N,n}_{s_n,\xi})\Big)\d s\nonumber\\&+\int_{0}^{t}\Big(b^{i,N,n}(t,s, X_{s_n,\xi}^{i,N,n}, \mu^{X,N,n}_{s_n,\xi})-b^{i,N,n}(\tilde{t},s,X_{s_n,\xi}^{i,N,n},\mu^{X,N,n}_{s_n,\xi})\Big)\d s\nonumber\\&+\int_{0}^{t}\Big(b^{i,N,n}(\tilde{t},s, X_{s_n,\xi}^{i,N,n}, \mu^{X,N,n}_{s_n,\xi})-b^{i,N,n}(\tilde{t},s_n,X_{s_n,\xi}^{i,N,n},\mu^{X,N,n}_{s_n,\xi})\Big)\d s\nonumber\\&+\int_{0}^{t}\Big(\sigma(t,s, X_{s,\xi}^{i,N}, \mu^{X,N}_{s,\xi})-\sigma(t,s,X_{s,\xi}^{i,N,n},\mu^{X,N,n}_{s,\xi})\Big)\d W_s\nonumber\\&+\int_{0}^{t}\Big(\sigma^{i,N,n}(t,s, X_{s,\xi}^{i,N,n}, \mu^{X,N,n}_{s,\xi})-\sigma^{i,N,n}(t,s,X_{s_n,\xi}^{i,N,n},\mu^{X,N,n}_{s_n,\xi})\Big)\d W_s\nonumber\\&+\int_{0}^{t}\Big(\sigma^{i,N,n}(t,s, X_{s_n,\xi}^{i,N,n}, \mu^{X,N,n}_{s_n,\xi})-\sigma^{i,N,n}(\tilde{t},s,X_{s_n,\xi}^{i,N,n},\mu^{X,N,n}_{s_n,\xi})\Big)\d W_s\nonumber\\&+\int_{0}^{t}\Big(\sigma^{i,N,n}(\tilde{t},s, X_{s_n,\xi}^{i,N,n}, \mu^{X,N,n}_{s_n,\xi})-\sigma^{i,N,n}(\tilde{t},s_n,X_{s_n,\xi}^{i,N,n},\mu^{X,N,n}_{s_n,\xi})\Big)\d W_s\nonumber\\&:=\Pi_1+\Pi_2+\Pi_3+\Pi_4+\Pi_5+\Pi_6+\Pi_7+\Pi_8.\nonumber
	\end{align}	
	For the term $\Pi_5$, by H$\mathrm{\ddot{o}}$lder's inequality, BDG's inequality, $\mathbf{(H1'')}$ and $\mathbf{(H2)}$, we have
	\begin{align}
		&E|\Pi_5|^p\nonumber\\&=E\Big|\int_{0}^{t}\Big(\sigma(t,s, X_{s,\xi}^{i,N}, \mu^{X,N}_{s,\xi})-\sigma(t,s,X_{s,\xi}^{i,N,n},\mu^{X,N,n}_{s,\xi})\Big)\d W_s\Big|^p\nonumber\\&\leq C_pE\Big|\int_{0}^{t}\Big(\sigma(t,s, X_{s,\xi}^{i,N}, \mu^{X,N}_{s,\xi})-\sigma(t,s,X_{s,\xi}^{i,N,n},\mu^{X,N,n}_{s,\xi})\Big)^2\d s\Big|^{\frac{p}{2}}\nonumber\\&\leq C_p E\Big|\int_{0}^{t}K_2(t,s)\cdot\Big(|X_{s,\xi}^{i,N}-X_{s,\xi}^{i,N,n}|^2+\mathbb{W}_2^2(\mu^{X,N}_{s,\xi}, \mu^{X,N,n}_{s,\xi})\Big)^2\d s\Big|^{\frac{p}{2}}\nonumber\\&\leq C_p \Big(\int_{0}^{t}K^{\a}_2(t,s)\d s\Big)^{\frac{p}{2\a}}\cdot E\Big|\int_{0}^{t}\Big(|X_{s,\xi}^{i,N}-X_{s,\xi}^{i,N,n}|^2+\mathbb{W}_2^2(\mu^{X,N}_{s,\xi}, \mu^{X,N,n}_{s,\xi})\Big)^{\a^*}\d s\Big|^{\frac{p}{2\a^*}}\nonumber\\&\leq C_p\int_{0}^{t}\Big(E|Z_{s,\xi}^{i,N,n}|^p+E\Big(\frac{1}{N}\sum_{i=1}^{N}|Z^{i,N,n}_{s,\xi}|^2\Big)^{\frac{p}{2}}\Big)\d s,\nonumber
	\end{align}
	where we have used the following elementary estimate
	$$\mathbb{W}_2^2(\mu^{X,N}_{s,\xi}, \mu^{X,N,n}_{s,\xi})\leq \frac{1}{N}\sum_{i=1}^{N}|X^{i,N}_{s,\xi}-X^{i,N,n}_{s,\xi}|^2.$$
	For the term $\Pi_6$, applying H$\mathrm{\ddot{o}}$lder's inequality, BDG's inequality, $\mathbf{(H1'')}$ and $\mathbf{(H2)}$, we get
	\begin{align}
		&E|\Pi_6|^p\nonumber\\&=\int_{0}^{t}\Big(\sigma^{i,N,n}(t,s, X_{s,\xi}^{i,N,n}, \mu^{X,N,n}_{s,\xi})-\sigma^{i,N,n}(t,s,X_{s_n,\xi}^{i,N,n},\mu^{X,N,n}_{s_n,\xi})\Big)\d W_s\nonumber\\&\leq C_pE\Big|\int_{0}^{t}\big\|\sigma^{i,N,n}(t,s, X_{s,\xi}^{i,N,n}, \mu^{X,N,n}_{s,\xi})-\sigma^{i,N,n}(t,s,X_{s_n,\xi}^{i,N,n},\mu^{X,N,n}_{s_n,\xi})\big\|^2\d s\Big|^{\frac{p}{2}}\nonumber\\&\leq C_p E\Big|\int_{0}^{t}K_2(t,s)\cdot\Big(|X_{s,\xi}^{i,N,n}-X_{s_n,\xi}^{i,N,n}|^2+\mathbb{W}_2^2(\mu^{X,N,n}_{s,\xi}, \mu^{X,N,n}_{s_n,\xi})\Big)^2\d s\Big|^{\frac{p}{2}}\nonumber\\&\leq C_p\int_{0}^{t}\Big(E|X_{s,\xi}^{i,N,n}-X_{s_n,\xi}^{i,N,n}|^p+E\Big(\frac{1}{N}\sum_{i=1}^{N}|X_{s,\xi}^{i,N,n}-X_{s_n,\xi}^{i,N,n}|^2\Big)^{\frac{p}{2}}\Big)\d s\nonumber\\&\leq C_p(1+|\xi|^p)|s-s_n|^{\theta p}\leq C_p(1+|\xi|^p)\cdot 2^{-n\theta p},\nonumber
	\end{align}
	where the last step we have used Lemma \ref{estimate of t}.
	
	For the term $\Pi_7$, by Lemma \ref{estimate 1}, BDG's inequality, extended Minkowski's inequality, $\mathbf{(H1')}$ and $\mathbf{(H4')}$, we have
	\begin{align}
		&E|\Pi_7|^p\nonumber\\&=E\Big|\int_{0}^{t}\Big(\sigma^{i,N,n}(t,s, X_{s_n,\xi}^{i,N,n}, \mu^{X,N,n}_{s_n,\xi})-\sigma^{i,N,n}(\tilde{t},s,X_{s_n,\xi}^{i,N,n},\mu^{X,N,n}_{s_n,\xi})\Big)\d W_s\Big|^p\nonumber\\&\leq C_p E \Big|\int_{0}^{t}\big\|\sigma^{i,N,n}(t,s, X_{s_n,\xi}^{i,N,n}, \mu^{X,N,n}_{s_n,\xi})-\sigma^{i,N,n}(\tilde{t},s,X_{s_n,\xi}^{i,N,n},\mu^{X,N,n}_{s_n,\xi})\big\|^2\d s\Big|^{\frac{p}{2}}\nonumber\\&\leq C_p E \Big|\int_{0}^{t}\big\|\sigma^{i,N,n}(t,s, X_{s_n,\xi}^{i,N,n}, \mu^{X,N,n}_{s_n,\xi})-\sigma^{i,N,n}(\tilde{t},s,X_{s_n,\xi}^{i,N,n},\mu^{X,N,n}_{s_n,\xi})\big\|^2\d s\Big|^{\frac{p}{2}}\nonumber\\&\leq C_p E\Big|\int_{0}^{t}F_2(t,\tilde{t},s)\cdot\big(1+|X_{s_n,\xi}^{i,N,n}|^2+\mathbb{W}^2_2(\mu^{X,N,n}_{s_n,\xi},\delta_0)\big)\d s\Big|^{\frac{p}{2}}\nonumber\\&\leq C_p\Bigg(\int_{0}^{t}F_2(t,\tilde{t},s)\cdot \big(E\big[1+|X^{i,N,n}_{s_n,\xi}|^p+\mathbb{W}_2^p(\mathcal{L}_{X^{i,N,n}_{s_n,\xi}},\delta_0)\big]\big)^{\frac{2}{p}}\d s \Bigg)^{\frac{p}{2}}\nonumber\\&\leq C_p(1+|\xi|^p)\Big(\int_{0}^{t}F_2(t,\tilde{t},s)\d s\Big)^{\frac{p}{2}}\leq C_p(1+|\xi|^p)|t'-t|^{\frac{\gamma p}{2}}\leq C_p(1+|\xi|^p)\cdot 2^{-\frac{n\gamma p}{2}}.\nonumber
	\end{align}
	For the term $\Pi_8$, by Lemma \ref{estimate 1}, BDG's inequality, extended Minkowski's inequality,  $\mathbf{(H1')}$ and $\mathbf{(H5)}$, we have
	\begin{align}
		&E|\Pi_8|^p\nonumber\\&=E\Big|\int_{0}^{t}\Big(\sigma^{i,N,n}(\tilde{t},s, X_{s_n,\xi}^{i,N,n}, \mu^{X,N,n}_{s_n,\xi})-\sigma^{i,N,n}(\tilde{t},s_n,X_{s_n,\xi}^{i,N,n},\mu^{X,N,n}_{s_n,\xi})\Big)\d W_s\Big|^p\nonumber\\&\leq C_p E\Big|\int_{0}^{t}\big\|\sigma^{i,N,n}(\tilde{t},s, X_{s_n,\xi}^{i,N,n}, \mu^{X,N,n}_{s_n,\xi})-\sigma^{i,N,n}(\tilde{t},s_n,X_{s_n,\xi}^{i,N,n},\mu^{X,N,n}_{s_n,\xi})\big\|^2\d s\Big|^{\frac{p}{2}}\nonumber\\&\leq C_p E\Big|\int_{0}^{t}F_4(\tilde{t},s,s_n)\cdot\Big(1+|X_{s_n,\xi}^{i,N,n}|^2+\mathbb{W}_2^2(\mu^{X,N,n}_{s_n,\xi},\delta_0)\Big)^2\d s\Big|^{\frac{p}{2}}\nonumber\\&\leq C_p(1+|\xi|^p)\Big(\int_{0}^{t}F_4(\tilde{t},s,s_n)\d s\Big)^{\frac{p}{2}}\leq C_p(1+|\xi|^p)\cdot 2^{-\frac{n\delta p}{2}}.\nonumber
	\end{align}
	$\Pi_1,\Pi_2,\Pi_3,\Pi_4$ can be similarly dealt with, i.e.
	\begin{align}
		&E|\Pi_1|^p\leq C_p\int_{0}^{t}\Big(E|Z_{s,\xi}^{i,N,n}|^p+E\Big(\frac{1}{N}\sum_{i=1}^{N}|Z^{i,N,n}_{s,\xi}|^2\Big)^{\frac{p}{2}}\Big)\d s,\nonumber \\&E|\Pi_2|^p\leq C_p(1+|\xi|^p)\cdot 2^{-n\theta p},\nonumber\\&E|\Pi_3|^p\leq C_p(1+|\xi|^p)\cdot 2^{-n\gamma p},\nonumber\\&E|\Pi_4|^p\leq C_p(1+|\xi|^p)\cdot 2^{-n\delta p}.\nonumber
	\end{align}
	Therefore, if we define
	$$g(t):=\sup_{i=1,\cdots,N}\sup_{s\in [0,t]} EZ_{s,\xi}^{i,N,n},$$
	then we easily obtain
	$$g(t)\leq C_p\int_{0}^{t}g(s)\d s+C_p(1+|\xi|^p)\cdot 2^{-n\eta p},$$
	where $\eta>0$ depends only $\a,\gamma,\delta$. Gronwall's inequality gives the desired result.
\end{proof}
\begin{thm}
	Assume that $\mathbf{(H1'')}$, $\mathbf{(H2)}$, $\mathbf{(H3)}$ ,$\mathbf{(H4')}$ and $\mathbf{(H5)}$ hold, there exists $\lambda>0$ such that for any $R>0$ and $p>2$ sufficiently large,
	$$E\Big(\sup_{t\in [0,T],|\xi|\leq R}\big|X_{t,\xi}^{i,N}-X_{t,\xi}^{i,N,n}\big|^p\Big)\leq E|A(\omega)|^p2^{-np\lambda}.$$
	In particular,
	$$P\Big\{\lim_{n\to\infty}\sup_{t\in [0,T],|\xi|\leq R}\big|X_{t,\xi}^{i,N}-X_{t,\xi}^{i,N,n}\big|=0\Big\}=1.$$
\end{thm}
\begin{proof}
	One constructs the following process:
	\begin{equation}
		Z(r,t,\xi)=\begin{cases}
			X^{i,N}_{t,\xi}, \ r=0,\nonumber\\
			Z^{i,N,n}_{t,\xi}+2^{n+1}(r-2^{-n})\big(X^{i,N,n+1}_{t,\xi}-X^{i,N,n}_{t,\xi}\big), \ 2^{-(n+1)}<r\leq 2^{-n},\ n\in\mathbb{N}. \nonumber
		\end{cases}
	\end{equation}
	Fix $R>0$, by Lemmas \ref{estimate 1}, \ref{estimate of t} and Theorem \ref{rate}, there exists a constant $\beta:=\beta(\gamma)>0$ such that for all $p>2$, $t,t'\in [0,T]$, $\xi,\eta\in D_R:=\{\xi\in\mathbb{R}^d,|\xi|\leq R\}$, we have
	$$E\Big|Z(r,t,\xi)-Z(r',t',\eta)\Big|^p\leq C_{p,T,R,L_3,L_4}\Big(|r-r'|^{\beta p}+|t-t'|^{\beta p}+|\xi-\eta|^{\beta p}\Big).$$
	Thus, by Kolmogorov's continuity criterium, there is a $p$-order integrable random variable $A(\omega)$ such that
	$$\sup_{t\in [0,T],|\xi|\leq R}\Big|Z(r,t,\xi)-Z(r',t,\xi)\Big|\leq A(\omega)|r-r'|^{\lambda},\quad a.s.,$$
	where $\lambda\in (0,\beta-\frac{d+2}{p})$. In particular, taking $r=0,r'=2^{-n}$, we have
	$$E\Big(\sup_{t\in [0,T],|\xi|\leq R}\big|X_{t,\xi}^{i,N}-X_{t,\xi}^{i,N,n}\big|\Big)^p\leq E|A(\omega)|^p 2^{-np\lambda},$$
	which yields the desired convergence. Borel-Cantelli's Lemma gives the second conclusion.
\end{proof}
\providecommand{\bysame}{\leavevmode\hbox to3em{\hrulefill}\thinspace}
\providecommand{\MR}{\relax\ifhmode\unskip\space\fi MR }
\providecommand{\MRhref}[2]{%
	\href{http://www.ams.org/mathscinet-getitem?mr=#1}{#2}
}
\providecommand{\href}[2]{#2}

\section*{Conflicts of interests}
The authors declare no conflict of interests.

\end{document}